\documentclass[a4paper,11pt,reqno]{amsart}

\usepackage[english]{babel}
\usepackage[headings]{fullpage}
\usepackage{amsmath} 
\usepackage{amssymb} 
\usepackage{amsopn}
\usepackage{amsthm}
\usepackage{amsfonts} 
\usepackage{mathrsfs}
\usepackage{mathtools}
\usepackage{stmaryrd}
\usepackage{manfnt} \usepackage{mathdots}
\usepackage[OT2, T1]{fontenc}
\usepackage[dvips]{graphicx} \usepackage[usenames]{color}
\usepackage[all]{xy}
\usepackage{dsfont}
\usepackage[utf8]{inputenc} 
\usepackage{ifthen}
\usepackage{marvosym}
\usepackage{wasysym}
\usepackage{marginnote}
\usepackage{float}
\usepackage{array,arydshln}
\usepackage{booktabs}
\usepackage{subfigure}
\usepackage{tikz}
\usepackage{tkz-euclide}
\usetikzlibrary{through,intersections,calc,angles}
\usetikzlibrary{shapes.geometric}
\usetikzlibrary{decorations.markings,calc}
\usepackage{pgfplots}
\usetikzlibrary{arrows,chains,positioning,matrix}
\usepackage{tikz-cd}
\usepackage{bbm}
\newboolean{printforcirculation}
\setboolean{printforcirculation}{false}
\newcommand{\hide}[1]{\ifthenelse{\boolean{printforcirculation}}{}{\cyan{#1}}}
\newcommand{\finalhide}[1]{\ifthenelse{\boolean{printforcirculation}}{}{#1}}
\usepackage{scalefnt}
\newboolean{bonusmaterial}
\setboolean{bonusmaterial}{true}
\newcommand{\bonus}[1]{\ifthenelse{\boolean{bonusmaterial}}{{\scalefont{0.8} #1}}}{}
\usepackage[pdfpagelabels=true, pdftex]{hyperref}
\hypersetup{
	pdftitle={},
	pdfauthor={},
	pdfsubject={},
	pdfkeywords={},
	colorlinks=true,    
	linkcolor=red,     
	citecolor=blue,     
	filecolor=blue,      
	urlcolor=blue,       
	breaklinks=true,
	bookmarksopen=true,
	bookmarksnumbered=true,
	pdfpagemode=UseOutlines,
	plainpages=false}

\DeclareRobustCommand{\gobblefive}[5]{}
\newcommand*{\SkipTocEntry}{\addtocontents{toc}{\gobblefive}}
\DeclareRobustCommand{\SkipTocEntry}[5]{}
\newcommand{\bN}{{\mathbb N}}

\newcommand{\bP}{{\mathbb P}}
\newcommand{\bQ}{{\mathbb Q}}
\newcommand{\bR}{{\mathbb R}}

\newcommand{\bZ}{{\mathbb Z}}

\newcommand{\dO}{{\mathcal O}}

\DeclareSymbolFont{cyrletters}{OT2}{wncyr}{m}{n}
\DeclareMathSymbol{\Sha}{\mathalpha}{cyrletters}{"58}

\DeclareMathOperator{\Aut}{Aut}

\DeclareMathOperator{\Hom}{Hom}
\DeclareMathOperator{\id}{id}

\DeclareMathOperator{\Isom}{Isom}

\newcommand{\xyinj}{\ar@{^(->}}

\DeclareMathOperator{\GL}{GL}

\DeclareMathOperator{\PGL}{PGL}

\newcommand{\tr}{{\rm tr}} 

\DeclareMathOperator{\ad}{ad}

\DeclareMathOperator{\Char}{char} 

\DeclareMathOperator{\Div}{Div}
\DeclareMathOperator{\divisor}{div}

\newcommand{\Nm}{{\rm Nm}}
 
\newcommand{\OO}{\dO}
\DeclareMathOperator{\PDiv}{PDiv}

\DeclareMathOperator{\Pic}{Pic}

\DeclareMathOperator{\Stab}{Stab}

\DeclareMathOperator{\trdeg}{{trdeg}}

\DeclareMathOperator{\Frob}{Frob}

\DeclareMathOperator{\ord}{ord}

\newcommand{\cd}{{\rm cd}}

\DeclareMathOperator{\cores}{cor}

\DeclareMathOperator{\Gal}{Gal}

\DeclareMathOperator{\res}{res}

\def\10{{\overrightarrow{10}}}
\def\01{{\overrightarrow{01}}}
\newcommand{\ab}{{\rm ab}}

\DeclareMathOperator{\Jac}{{\rm Jac}}

\newcommand{\op}{{\rm op}}

\newcommand{\sep}{{\rm sep}}

\newcommand{\tot}{{\rm tot}}

\newcommand{\lto}{\longrightarrow}
\newcommand{\Map}{{\rm Map}}

\def\proofname{Beweis}

\newtheorem{thm}{\protect\thrmname}

\newtheorem{prop}[thm]{\protect\propname}
\newtheorem{lem}[thm]{\protect\lemname}

\newtheorem{thmalpha}{\protect\thmalphaname}

\theoremstyle{definition}
\newtheorem{defi}[thm]{\protect\definame}

\theoremstyle{remark}
\newtheorem{rmk}[thm]{\protect\rmkname}

\newtheorem*{rmks*}{\protect\rmksname}

\newenvironment{pro*}[1][\proofname]{{\it{#1:}} }{}
\newenvironment{pro**}[1][]{{\it{#1}} }{\hfill $\square$}

\numberwithin{equation}{section}
\numberwithin{thm}{section}

\newcommand{\thrmname}{}
\newcommand{\thmdefiname}{}
\newcommand{\propname}{}
\newcommand{\lemname}{}
\newcommand{\lemdefiname}{}
\newcommand{\conjname}{}
\newcommand{\factname}{}
\newcommand{\corname}{}
\newcommand{\definame}{}
\newcommand{\axiomname}{}
\newcommand{\algoname}{}
\newcommand{\rmkname}{}
\newcommand{\rmksname}{}
\newcommand{\notaname}{}
\newcommand{\exname}{}
\newcommand{\questionname}{}
\newcommand{\applname}{}
\newcommand{\claimname}{}
\newcommand{\exercisename}{}
\newcommand{\thmalphaname}{}

\addto\captionsenglish{
	\renewcommand{\thrmname}{Theorem}
	\renewcommand{\thmdefiname}{Theorem--Definition}
	\renewcommand{\propname}{Proposition}
	\renewcommand{\lemname}{Lemma}
	\renewcommand{\lemdefiname}{Lemma--Definition}
	\renewcommand{\conjname}{Conjecture}
	\renewcommand{\factname}{Fact}
	\renewcommand{\corname}{Corollary}
	\renewcommand{\definame}{Definition}
	\renewcommand{\axiomname}{Axiom}
	\renewcommand{\algoname}{Algorithm}
	\renewcommand{\rmkname}{Remark}
	\renewcommand{\notaname}{Notation}
	\renewcommand{\exname}{Example}
	\renewcommand{\questionname}{Question}
	\renewcommand{\applname}{Application}
	\renewcommand{\claimname}{Claim}
	\renewcommand{\exercisename}{Exercise}
	\renewcommand{\thmalphaname}{Theorem}
	\renewcommand{\rmksname}{Remarks}
}
\addto\captionsngerman{%
	\renewcommand{\thrmname}{Satz}
	\renewcommand{\thmdefiname}{Satz--Definition}
	\renewcommand{\propname}{Proposition}
	\renewcommand{\lemname}{Lemma}
	\renewcommand{\lemdefiname}{Lemma--Definition}
	\renewcommand{\conjname}{Vermutung}
	\renewcommand{\factname}{Fakt}
	\renewcommand{\corname}{Korollar}
	\renewcommand{\definame}{Definition}
	\renewcommand{\axiomname}{Axiom}
	\renewcommand{\algoname}{Algorithmus}
	\renewcommand{\rmkname}{Bemerkung}
	\renewcommand{\notaname}{Notation}
	\renewcommand{\exname}{Beispiel}
	\renewcommand{\questionname}{Frage}
	\renewcommand{\applname}{Anwendung}
	\renewcommand{\claimname}{Behauptung}
	\renewcommand{\exercisename}{Übung}
	\renewcommand{\thmalphaname}{Satz}
	\renewcommand{\rmksname}{Bemerkungen}
}

\newcounter{absatzcounter}[section]
\usepackage{enumitem}
\newlist{enumer}{enumerate}{2}
\setlist[enumer]{label=(\roman*),align=left,labelindent=.5cm,leftmargin=*,widest = (iii)}
\newlist{enumerar}{enumerate}{1}
\setlist[enumerar]{label=\arabic*.,align=left,labelindent=.5cm,leftmargin=*,widest = 8.}
\newlist{enumera}{enumerate}{2}
\setlist[enumera]{label=(\arabic*),align=left,labelindent=.5cm,leftmargin=*,widest = (8)}
\newlist{enumeral}{enumerate}{2}
\setlist[enumeral]{label=(\alph*),align=left,labelindent=.5cm,leftmargin=*,widest = (m)}

\newcommand{\makecyan}{\color{cyan}}

\newcommand{\cyan}[1]{{\makecyan #1}}

\definecolor{shadecolor}{RGB}{186,238,186}
\definecolor{softlimegreen}{RGB}{186,238,186}
\definecolor{limegreen}{RGB}{208,243,208}
\definecolor{questioncolor}{RGB}{135, 173, 241} 
\definecolor{warningcolor}{RGB}{240,120,134}
\definecolor{verypaleyellow}{RGB}{255,255,194}

\newenvironment{diagram}{\begin{center}\begin{tikzcd}}{\end{tikzcd}\end{center}}

\newcommand*\isomto{%
	\mathrel{\vbox{\offinterlineskip\ialign{%
		\hfil##\hfil\cr
		$\scriptscriptstyle\sim$\cr
		\noalign{\kern0.1ex}
		$\longrightarrow$\cr
	}}}}

\newcommand\suchthat{%
	\mathrel{}\middle|\mathrel{}
}

\newcommand{\defstyle}[1]{\textbf{#1}}

\newcommand*{\longhookrightarrow}{\ensuremath{\lhook\joinrel\relbar\joinrel\rightarrow}}

\begin{document}
	
\title[A Birational Anabelian Reconstruction Theorem for Curves]{A Birational Anabelian Reconstruction Theorem for Curves over Algebraically Closed Fields in Arbitrary Characteristic}
\author{Martin Lüdtke}
\address{Institut für Mathematik\\Johann Wolfgang Goethe-Universität\\Robert-Mayer-Str. 6--8\\60325 Frankfurt am Main\\Germany}
\email{luedtke@math.uni-frankfurt.de}

\begin{abstract}
	The aim of Bogomolov's programme is to prove birational anabelian conjectures for function fields $K|k$ of varieties of dimension $\geq 2$ over algebraically closed fields. The present article is concerned with the 1-dimensional case. While it is impossible to recover $K|k$ from its absolute Galois group alone, we prove that it can be recovered from the pair $(\Aut(\overline{K}|k),\Aut(\overline{K}|K))$, consisting of the absolute Galois group of $K$ and the larger group of field automorphisms fixing only the base field. 
\end{abstract}
\maketitle
\tableofcontents

\section{Introduction}

The aim of the birational anabelian program initiated by Bogomolov~\cite{bogomolov} at the beginning of the 1990's is to recover function fields $K|k$ of dimension $>1$ over algebraically closed fields from their absolute Galois group $G_K$. This cannot be possible in the one-dimensional case since then $G_K$ is profinite free of rank $|k|$ by results of Harbater~\cite{harbater} and Pop~\cite{pop-galois-covers}, containing therefore almost no information about $K$. We show however that $K|k$ can be recovered if in addition to $G_K$ also the larger automorphism group $\Aut(\overline{K}|k) \supseteq G_K$ fixing only the base field is provided. On the way, we prove a Galois-type correspondence for transcendental field extensions and give a group-theoretic characterisation of stabiliser subgroups for $\PGL(2,k)$ acting on $\bP^1$.

We use the following notation: Let $k$ be an algebraically closed base field. A function field $K|k$ is a finitely generated field extension. Its dimension is defined as the transcendence degree $\trdeg(K|k)$. If $F := \overline{K}$ is an algebraic closure, we write $G_{F|k} := \Aut(F|k)$ and denote the absolute Galois group of $K$ by $U_K := \Aut(F|K) \cong \Gal(K^{\sep}|K)$. The group $G_{F|k} \subseteq \Map(F,F)$ is endowed with the compact-open topology for discrete $F$, making $U_K$ an open subgroup whose induced topology agrees with the usual profinite Krull topology.

\begin{thmalpha}
	\label{thm-function-field-main-theorem}
	Let $k$ be an algebraically closed field and $K|k$ a 1-dimensional function field with algebraic closure $F = \overline{K}$. If $(k',K',F')$ is another such triple, then the natural map
	$$\Phi: \Isom(F'|K'|k', F|K|k) \longrightarrow \Isom((G_{F|k},U_K), (G_{F'|k'},U_{K'}))$$
	is a bijection.
\end{thmalpha}

The right hand side consists of isomorphisms of topological groups $G_{F|k} \isomto G_{F'|k'}$ which restrict to an isomorphism between the open subgroups $U_{K} \cong U_{K'}$. An isomorphism between field towers $F'|K'|k'$ and $F|K|k$ is by definition an isomorphism $\sigma: F' \isomto F$ restricting to isomorphisms $K' \cong K$ and $k' \cong k$. The natural map in the theorem assigns to such $\sigma$ the isomorphism $\Phi(\sigma): (G_{F|k},U_K) \isomto (G_{F'|k'},U_{K'})$ given by
$$\Phi(\sigma)(\tau) = \sigma^{-1} \circ \tau \circ \sigma.$$

It is useful to not single out one function field in $F|k$ but rather work with the totality of them and study their interplay that comes from inclusions between them, or equivalently from morphisms between their complete nonsingular models. We therefore prove the following variant which, as proved at the end of Section~\ref{sec-galois-type-correspondence}, implies Theorem~\ref{thm-function-field-main-theorem}.

\begin{thmalpha}
	\label{thm-main-theorem}
	Let $F|k$ be an extension of algebraically closed fields of transcendence degree 1. If $F'|k'$ is another extension of algebraically closed fields, the natural map
	\begin{align*}
		\Phi: \Isom^i(F'|k', F|k) &\lto \Isom(G_{F|k},G_{F'|k'})
	\end{align*}
	is a bijection.
\end{thmalpha}

Here, $\Isom^i$ denotes the set of isomorphisms up to Frobenius twists, i.e.\ identifying $\sigma \sim \Frob^n \circ \sigma$ for $n \in \bZ$ if $\Char(k) = \Char(k') = p > 0$, where $\Frob \in \Isom(F|k,F|k)$ is $\Frob(x) = x^p$.

The present article is heavily based on~\cite{rovinsky}. It contains however several new aspects, most notably extending the results to function fields of positive characteristic. This required a Galois-type correspondence theorem for certain transcendental field extensions (Theorem~\ref{thm-galois}). Moreover, we include the details of how the stabiliser subgroups of $\PGL(2,k)$ acting on $\bP^1$ are group-theoretically distinguished (Lemma~\ref{lem-stab-subgroups-in-pgl}) and present a simplified way of detecting decomposition groups (Lemma~\ref{lem-detect-dec-groups}). We also take an alternative route to reconstruct the function fields from there, namely via the reconstruction of ramification indices and principal divisors, whereas the approach of \cite{rovinsky} is based on linear systems.

\subsection*{Acknowledgements}
The author would like to thank Jakob Stix for careful proofreading and many helpful comments, Armin Holschbach for numerous discussions and Alexander Schmidt for supervising the master's thesis from which this article originated.

\section{Injectivity}
\label{sec-injectivity}

We start by quickly treating the easier part of Theorem~\ref{thm-function-field-main-theorem}, namely the question of injectivity. In fact, we show the following stronger statement.

\begin{thm}
	\label{thm-injectivity}
	Let $K$ be a function field of dimension $\geq 1$ over an algebraically closed field $k$, let $F$ be its algebraic closure and $U_K = \Aut(F|K)$ its absolute Galois group. If $(k',K',F')$ is another such triple, then the natural map
	$$\Isom(F'|K', F|K) \lto \Isom(U_K, U_{K'})$$
	is injective.
\end{thm}
Note that the function fields $K$ and $K'$ are allowed to have any dimension $\geq 1$, and on the right hand side we have isomorphisms between $U_K$ and $U_{K'}$ rather than isomorphisms of pairs $(G_{F|k},U_K)$, $(G_{F'|k'},U_{K'})$. Thus the statement applies also to situations that Bogomolov's conjecture is concerned with. Our proof uses a valuation theoretic result of F.~K.~Schmidt and is similar to that of the last lemma in \cite{pop-galois-theory-of-function-fields}.

Recall that a \defstyle{rank 1 valuation} on a field $K$ is a map $v: K \to \bR \cup \{\infty\}$ satisfying
\begin{enumerate}[leftmargin=1.5cm,label=(\alph*),itemsep=0ex]
	\item{$v(x) = \infty \; \Leftrightarrow \; x = 0;$}
	\item{$v(xy) = v(x) + v(y);$}
	\item{$v(x+y) \geq \min\{v(x), v(y)\};$}
\end{enumerate}
and moreover $v \not\equiv 0$ on $K^\times$. It is a \defstyle{discrete valuation} if its value group $v(K^\times)$ is a discrete subgroup of $\bR$. Two rank 1 valuations $v$, $w$ on $K$ are \defstyle{equivalent}, written $v \sim w$, if $w = \lambda v$ for some $\lambda > 0$. The automorphism group $\Aut(K)$ acts on the set of rank 1 valuations on $K$ by the rule $\tau v = v \circ \tau^{-1}$. The \defstyle{decomposition group} of $v$ over a subfield $k \subseteq K$ is defined as
$$Z_v(K|k) = \left\{ \tau \in \Aut(K|k) \suchthat \tau v \sim v \right\}.$$
The field $K$ is \defstyle{henselian} with respect to the valuation $v$ if $v$ extends uniquely to every algebraic extension of $K$.

\begin{lem}
	\label{lem-dec-groups-trivial intersection}
	Let $K$ be a field, $F = \overline{K}$ its algebraic closure and $v$, $w$ two rank 1 valuations on $F$ with decomposition groups $Z_v$, $Z_w$ over $K$. If $v$ and $w$ are inequivalent, then
	$$Z_v \cap Z_w = 1.$$
\end{lem}

\begin{proof}
	The fixed field $F^{Z_v \cap Z_w}$ is henselian with respect to the restrictions of $v$ and $w$, hence is separably closed by a theorem of F.~K.~Schmidt (cf.~\cite{engler-prestel}, Thm.~4{.}4{.}1). Thus, $F^{Z_v \cap Z_w} = F$ and $Z_v \cap Z_w = 1$.
\end{proof}

\begin{lem}
	\label{lem-autom-id-citerion}
	Let $K$ be a field and $\sigma \in \Aut(K)$. If there exists a proper subfield $k \subset K$ such that $\sigma x / x \in k^\times$ for all $x \in K^\times$, then $\sigma = \id$.
\end{lem}

\begin{proof}
	Put $a(x) = \sigma x / x \in k^\times$. By additivity of $\sigma$, we have
	$$a(x) x + a(y) y = a(x+y) x + a(x+y) y$$
	for all $x,y \in K^\times$ with $x+y \neq 0$. For $k$-linearly independent $x$ and $y$, we obtain $a(x) = a(x+y) = a(y)$. For linearly dependent $x$ and $y$ we find $a(x) = a(y)$ as well, by comparing with an element of $K^\times$ which is linearly independent from $x$ and $y$. So we have $\sigma x = ax$ for some $a \in k^\times$ which does not depend on $x$. Since $\sigma(1) = 1$, we have $a = 1$.
\end{proof}

\begin{proof}[Proof of Theorem~\ref{thm-injectivity}]
	It suffices to show that if $\sigma \in \Aut(F)$ is an automorphism with $\sigma K = K$ for which $\sigma^{-1} (-)\sigma = \id$ in $\Aut(U_K)$, then $\sigma = \id$. Consider a rank 1 valuation $v$ of $F$ which is discrete on $K$. Its decomposition group $Z_v$ over $K$ is nontrivial since $v$ ramifies in $K^\sep|K$, e.g.\ by adjoining $\ell$-th roots of a uniformiser for $\ell \neq \Char(k)$. The decomposition group of $\sigma^{-1}v$ is $Z_{\sigma^{-1}v} = \sigma^{-1}Z_v \sigma = Z_v$, hence $\sigma^{-1}v = \lambda v$ for some $\lambda > 0$ by Lemma~\ref{lem-dec-groups-trivial intersection}. Since $\sigma K = K$, the valuations $v$ and $\lambda v$ have the same value group on $K$. As this value group is discrete, we must have $\lambda = 1$, thus $\sigma^{-1}v = v$. So we have $v(\sigma x / x) = 0$ for all $x \in F^\times$ and all rank 1 valuations $v$ on $F$ which are discrete on $K$. We claim that this implies $\sigma x / x \in k^\times$, so that we are done by Lemma~\ref{lem-autom-id-citerion}. 
	Indeed, otherwise we find a transcendence basis $T = (t_1,\ldots,t_n)$ of $F|k$ with $t_1 = \sigma x/x$ and a rank 1 valuation $v$ on $F$ extending the $t_1$-adic valuation on $k(T)$. Then $v$ is discrete on $k(T)$ and since the composite field $K(T)$ is a common finite extension of $K$ and $k(T)$, it is also discrete on $K$. But by construction $v(\sigma x/x) \neq 0$, contradiction!
\end{proof}

\section{A Galois-Type Correspondence for Transcendental Field Extensions}
\label{sec-galois-type-correspondence}
Let $k$ be a field and $F|k$ a field extension with $F$ algebraically closed. We prove a generalised Galois correspondence for such extensions and apply it in the case where $k$, too, is algebraically closed and $\trdeg(F|k) = 1$. Let $G_{F|k} = \Aut(F|k)$ be endowed with the compact-open topology for discrete $F$, so that a basis of open neighbourhoods of the identity is given by the subgroups $\Aut(F|K)$ with $K|k$ a finitely generated subextension of $F|k$. For such $K$, we denote the open subgroup $\Aut(F|K) \leq G_{F|k}$ by $U_K$. The group $G_{F|k}$ is a Hausdorff and totally disconnected topological group, since $\sigma \in U_{k(x)}$ for all $x \in F$ implies $\sigma = \id$.

In positive characteristic, one has to deal with the phenomenon that purely inseparable extensions are not visible in field automorphism groups. We therefore consider an equivalence relation on the set of subfields of $F$, which we call perfect equivalence.

\begin{defi}
	\label{def-purely-inseparable-equivalence}
	For a subfield $L$ in $F$, its \defstyle{perfect closure} $L^i$ in $F$ consists of all elements in $F$ that are purely inseparable over $L$. If $\Char k = p > 0$, it is given by
	$$L^i = L^{p^{-\infty}} = \left\{ x \in F : x^{p^n} \in L \text{ for some } n\in \bN \right\}.$$
	We call two subfields $L_1,L_2$ of $F$ \defstyle{perfectly equivalent} if $L_1^i = L_2^i$.
\end{defi}

In characteristic $p > 0$, since $p$th roots are unique, the image of $x \in F$ under an automorphism is uniquely determined by the image of $x^{p^n}$, hence $\Aut(F|L) = \Aut(F|L^i)$ for all subfields $L$ of $F$. Automorphism groups of field extensions are therefore "blind" towards purely inseparable extensions in the sense that perfectly equivalent subfields $L_1$ and $L_2$ satisfy $\Aut(F|L_1) = \Aut(F|L_2)$. However, a subfield $L$ of $F$ can be recovered from $\Aut(F|L)$ up to perfect equivalence. 

\begin{lem}
	\label{lem-fixed-field}
	Let $F$ be an algebraically closed field. Then for all subfields $L \subseteq F$, we have $F^{\Aut(F|L)} = L^i$.
\end{lem}

\begin{proof}
	The inclusion $(\supseteq)$ is trivial. If $x \in F$, but $x \not\in L^i$, there exists some $x' \in F \setminus \{x\}$ and an isomorphism $L(x) \cong L(x')$ over $L$, sending $x$ to $x'$; take $x' = x+1$ if $x$ is transcendental over $L$, and any root $x' \neq x$ of the (not purely inseparable) minimal polynomial of $x$ over $L$ if $x$ is algebraic. Since $F|L(x)$ and $F|L(x')$ have equal transcendence degree (possibly infinite), the isomorphism $L(x) \cong L(x')$ extends to an automorphism $\sigma$ of $F$. We have $\sigma \in \Aut(F|L)$, but $\sigma(x) \neq x$, therefore $x \not\in F^{\Aut(F|L)}$.
\end{proof}

\begin{thm}[Galois-Type Correspondence]
	\label{thm-galois}
	Let $F|k$ be a field extension with $F$ algebraically closed and let $G_{F|k} = \Aut(F|k)$. Then the map $L \mapsto \Aut(F|L)$ is injective up to perfect equivalence and restricts to bijections as follows:
	
	\begin{center}
		\begin{tikzpicture}[commutative diagrams/every diagram]
		\node (A1)
		{ $\left\{ \parbox{4.5cm}{\centering subfields $L$ in $F|k$, up to perfect equivalence } \right\}$ };
		
		\node (A2) [right = of A1]
		{ $\left\{ \parbox{4.3cm}{\centering closed subgroups of $G_{F|k}$ } \right\}$ };
		
		\node (AB1) [below = 0.15cm of A1, rotate = 90, xscale=1.8] {$\subseteq$};
		\node (AB2) [below = 0.5cm of A2, rotate = 90, xscale=2] {$\subseteq$};
		
		\node (B1) [below = 0.4cm of A1]
		{ $\left\{ \parbox{4.5cm}{\centering subfields $L$ in $F|k$ with $\overline{L} = F$, up to perfect equivalence } \right\}$ };
		
		\node (B2) [right = of B1]
		{ $\left\{ \parbox{4.3cm}{\centering compact subgroups of $G_{F|k}$ } \right\}$ };
		
		\node (BC1) [below = 0.25cm of B1, rotate = 90, xscale=1.8] {$\subseteq$};
		\node (BC2) [below = 0.7cm of B2, rotate = 90, xscale=2] {$\subseteq$};
		
		\node (C1) [below = 0.6cm of B1]
		{ $\left\{ \parbox{4.5cm}{\centering finitely generated subfields $L$ in $F|k$ with $\overline{L} = F$, up to perfect equivalence } \right\}$ };
		
		\node (C2) [right = of C1]
		{ $\left\{ \parbox{4.3cm}{\centering compact open subgroups of $G_{F|k}$ } \right\}$ };
		
		\path[commutative diagrams/.cd, every arrow, every label]
		(A1) edge[commutative diagrams/hook] (A2)
		(B1) edge node {$\sim$} (B2)
		(C1) edge node {$\sim$} (C2);
		
		\end{tikzpicture}
	\end{center}
\end{thm}

\begin{proof}
	Write $G := G_{F|k}$. For every subfield $L$ in $F|k$, the group
	$$\Aut(F|L) = \bigcap_{x \in L} U_{k(x)}$$
	is closed in $G$. If $F|L$ is algebraic, then the $\Aut(F|L)$-orbit of every $x \in F$ is finite and the product $\prod (X-x_i)$ with $x_i$ running through this orbit is a separable polynomial annihilating $x$ with coefficients in $F^{\Aut(F|L)} = L^i$, so that $F$ is separable and hence Galois over $L^i$. Thus $\Aut(F|L) = \Gal(F|L^i)$ is compact if $\overline{L} = F$, and compact open if in addition $L$ is finitely generated over $k$.
	
	Suppose $H$ is a compact subgroup of $G$. Then for every $x \in F$, the orbit $Hx$ is compact and discrete, hence finite, so $x$ is a root of a separable polynomial with coefficients in $F^H$. Thus $F|F^H$ is Galois and we have $H = \Aut(F|H)$ by Galois theory. This establishes the middle bijection.
	
	It remains to show that if $H \leq G$ is a compact open subgroup, then there exists a finitely generated subfield $K$ in $F|k$ with $H = \Aut(F|K)$. Since the sets $U_K$ with $K|k$ finitely generated form a neighbourhood basis of the identity, there exists such $K$ with $U_K \subseteq H$. Taking fixed fields, we get $F^H \subseteq K^i$. As $F^H|F^H \cap K$ is purely inseparable and $H$ is compact, we have
	$$\Aut(F|F^H \cap K) = \Aut(F|F^H) = H,$$
	and $F^H \cap K$ is finitely generated over $k$, being contained in the finitely generated extension $K|k$.
\end{proof}

\begin{rmks*} 
	\begin{enumeral} [itemsep=1ex]
		\item{The association $L \mapsto \Aut(F|L)$ is compatible with the $G_{F|k}$-actions in the sense that
			\begin{equation}
				\label{eq-conj-corresp}
				\Aut(F|\sigma L) = \sigma \Aut(F|L) \sigma^{-1}
			\end{equation}
			for all subextensions $L$ of $F|k$ and all $\sigma \in G_{F|k}$. Moreover, the $G_{F|k}$-action on the set of subextensions of $F|k$ is compatible with perfect equivalence since $\sigma(L)^i = \sigma(L^i)$.}
		\item{
			\label{rem-normaliser-quotient}
			Let $L$ be a subextension of $F|k$ and $H = \Aut(F|L)$. Then the normaliser of $H$ in $G_{F|k}$ is given by
			$
			N_{G_{F|k}}(H) = \left\{ \sigma \in G_{F|k} \suchthat \sigma L^i = L^i \right\}
			$
			and the restriction homomorphism $N_{G_{F|k}}(H) \to \Aut(L^i|k)$ induces an isomorphism of topological groups
			$$ N_{G_{F|k}}(H)/H \cong \Aut(L^i|k).$$}
		\item{The Galois correspondence is inclusion-reversing in the sense that
			$$L_1^i \subseteq L_2^i \; \Leftrightarrow \; \Aut(F|L_2) \subseteq \Aut(F|L_1)$$
			for all subfields $L_1$ and $L_2$ of $F|k$.}
		\item{If $L_1 \subseteq L_2$ is an algebraic extension of subfields of $F|k$, the index $(\Aut(F|L_1) : \Aut(F|L_2))$ equals the separable degree $[L_2 : L_1]_s$ since the left cosets are in canonical bijection with the $L_1$-embeddings $L_2 \hookrightarrow \overline{L_1}$.}
		\item{If the transcendence degree of $F|k$ is finite, then every finitely generated subextension $K$ is contained in a finitely generated extension $L$ with $\overline{L} = F$. Otherwise, there are no subextensions $L$ that are both finitely generated over $k$ and have $\overline{L} = F$. Consequently, $G_{F|k}$ is locally compact if and only if $F|k$ has finite transcendence degree. If moreover $k$ algebraically closed, the transcendence degree of $F|k$ can be recovered: let $U_K \leq G_{F|k}$ be any compact open subgroup, corresponding to a finitely generated subfield $K|k$ with $\overline{K} = F$. Then $\trdeg(F|k) = \cd_\ell(U_K)$ for all primes $\ell \neq \Char(k)$ (\cite{serre-galois-coh}, Ch.~II, Proposition~11). \label{recover_transcendence_degree}}
		\item{In general, there exist closed subgroups of $G \coloneqq G_{F|k}$ that do not arise as $\Aut(F|K)$ for some subfield $K$ of $F|k$. A subgroup $H \leq G$ arises in this way if and only if the inclusion $H \subseteq \Aut(F|F^H)$ is an equality. For a counterexample, consider the closed subgroup $G^\circ \subseteq G$ topologically generated by all compact subgroups. If $T$ is a transcendence basis of $F|k$, so is $T^n := \{t^n\}_{t \in T}$ for $n \geq 1$. Thus, the groups $\Aut(F|k(T^n))$ are all compact, hence $F^{G^\circ} = \bigcap_n k(T^n)^i = k^i$. However, if $F|k$ has finite transcendence degree $\geq 1$ with $x \in F$ transcendental over $k$, then $G^\circ$ is a proper subgroup of $G$ as $G^\circ$ is unimodular but any automorphism $\sigma \in G$ extending $k(x) \cong k(x^2)$ has Haar modulus $\Delta(\sigma) = 2$. 
			
		One can show that the subgroups of the form $\Aut(F|K) \subseteq G$ are stable under passage to closed supergroups with compact quotient. For subgroups this is false: assuming $1 \leq \trdeg(F|k) < \infty$, the Haar modulus induces a surjective homomorphism $\Delta: G/G^\circ \twoheadrightarrow \bQ_{>0}$, yielding many finite index subgroups $H \subsetneq G$ containing $G^\circ$. They all satisfy $F^H = k^i$ and hence $\Aut(F|F^H) = G$, so they are not of the form $\Aut(F|K)$.}
	\end{enumeral}
\end{rmks*}

Parts of the Galois-type correspondence appear in the literature as~\cite{jacobson}, p.\ 151, Exercise~5; \cite{shimura}, Propositions~6{.}11 and 6{.}12; \cite{shapiro-shafarevic}, §3, Lemma~1. A statement very close to ours in that it encompasses the case of positive characteristic is contained in \cite{rovinsky-motives}, Appendix~B, under the slightly stronger assumptions that $k$ be algebraically closed and that $\trdeg(F|k)$ be countable and $\geq 1$.


We now apply the Galois-type correspondence to the situation at hand where $k$ is algebraically closed and $\trdeg(F|k) = 1$. When we use the term \defstyle{function field in $F$}, we shall always mean one of dimension $1$, in other words we exclude the trivial function field $k$. The Galois-type correspondence shows that the function fields $K|k$ in $F$ are encoded in $G_{F|k}$ as the compact open subgroups $U_K$, up to perfect equivalence. By Remark~\ref{recover_transcendence_degree} above, the transcendence degree of $F|k$ is encoded as the $\ell$-cohomological dimension for $\ell \neq \Char(k)$ of any such $U_K$.

\begin{prop}
	\label{prop-function-field-bijection}
	In the situation of Theorem~\ref{thm-main-theorem}, let $\lambda: G_{F|k} \isomto G_{F'|k'}$ be an isomorphism. Then also $\trdeg(F'|k') = 1$ and $\lambda$ induces a bijection
	\begin{center}
		\begin{tikzpicture}
		\node (A) {$\left\{ \parbox{4.5cm}{\centering function fields $K'|k'$ in $F'$, up to perfect equivalence } \right\}$};
		\node (B) [right = of A] {$\left\{ \parbox{4.5cm}{\centering function fields $K|k$ in $F$, up to perfect equivalence } \right\}$};
		\path[commutative diagrams/.cd, every arrow, every label] (A) edge node {$\sim$} (B);
		\end{tikzpicture}
	\end{center}
	given by $K' \mapsto K$ whenever $\lambda^{-1}(U_{K'}) = U_K$. \qed
\end{prop}

We have the an explicit description of perfect equivalence for 1-dimensional function fields.
\begin{prop}[\cite{hartshorne} IV, Proposition~2.5]
	Let $F|k$ be an extension of algebraically closed fields of characteristic $p > 0$ and let $K|k$ be a 1-dimensional function field in $F$. Then for each $n \in \bN_0$, the extension $K^{p^{-n}}|K$ is the unique purely inseparable extension of $K$ in $F$ of degree $p^n$. Moreover, every function field in $F$ perfectly equivalent to $K$ is of the form $K^{p^{n}}$ for some $n \in \bZ$. In particular, they form an infinite field tower
	\label{prop-purely-inseparable-equivalence}
	\[
	\pushQED{\qed} 
	\ldots \subset K^{p^2} \subset K^p \subset K \subset K^{p^{-1}} \subset K^{p^{-2}} \subset \ldots. \qedhere
	\popQED
	\]
\end{prop}

Recall that for an algebraic field extension $L|K$ with relative separable closure $K \subseteq K_s \subseteq L$, the \defstyle{inseparable degree} of $L|K$ is defined as $[L : K]_i \coloneqq [L : K_s]$. 

\begin{defi}
	\label{def-gen-insep-degree}
	Let $F|k$ be an extension of algebraically closed fields and $K_1,K_2$ one-dimensional function fields in $F$ with $K_1^i \subseteq K_2^i$ (but not necessarily $K_1 \subseteq K_2$). If $\Char(k) = p > 0$, we define the \defstyle{generalised inseparable degree} $[K_2 : K_1]_i \in p^{\bZ}$ to be $p^{-n}[K_2 : K_1^{p^n}]_i$ where $n$ is sufficiently large such that $K_1^{p^n} \subseteq K_2$. If $\Char(k) = 0$, we set $[K_2 : K_1]_i = 1$.
\end{defi}

We note the following properties of the generalised inseparable degree:
\begin{enumeral}
	\item{For $K_1^i \subseteq K_2^i \subseteq K_3^i$ one has $[K_3 : K_1]_i = [K_3 : K_2]_i [K_2:K_1]_i$.}
	\item{$K_1 \subseteq K_2$ iff $[K_2: K_1]_i \geq 1$.}
	\item{$K_2|K_1$ is separable iff $[K_2:K_1]_i = 1$.}
\end{enumeral}

For a function field $K|k$ in $F$, we can recover the automorphism group $\Aut(K^i|k)$ from $(G_{F|k},U_K)$ as the quotient $N_{G_{F|k}}(U_K)/U_K$ by Remark~\ref{rem-normaliser-quotient}. It is related to $\Aut(K|k)$ as follows.

\begin{lem}
	\label{lem-aut-Ki}
	Let $K$ be a one-dimensional function field over an algebraically closed field $k$. Then there is a canonical exact sequence
	$$1 \lto \Aut(K|k) \lto \Aut(K^i|k) \lto \bZ.$$
\end{lem}

\begin{proof}
	The statement is trivial in characteristic zero, so assume $\Char(k) = p > 0$. Every automorphism of $K|k$ extends uniquely to $K^i$, whence the injective homomorphism $\Aut(K|k) \hookrightarrow \Aut(K^i|k)$. The second map is defined as $\sigma \mapsto \log_p\, [\sigma(K) : K]_i$ and the exactness is readily checked.
\end{proof}

\begin{proof}[Proof of Theorem~\ref{thm-main-theorem} $\Rightarrow$ Theorem~\ref{thm-function-field-main-theorem}]
	Consider the diagram
	\begin{diagram}
		\Isom(F'|K'|k, F|K|k) \dar[hook] \rar{\Phi}[swap]{(A)} & \Isom((G_{F|k}, U_K), (G_{F'|k'}, U_{K'})) \dar[hook] \\
		\Isom^i(F'|k', F|k) \rar{\Phi}[swap]{(B)} & \Isom(G_{F|k}, G_{F'|k'}).
	\end{diagram}
	The left vertical map is injective because the condition $\sigma(K') = K$ determines $\sigma$ uniquely among its Frobenius twists. Theorems~\ref{thm-function-field-main-theorem} and \ref{thm-main-theorem} assert the bijectivity of the top and bottom horizonal map, respectively. The square is cartesian: For $\sigma \in \Isom(F'|k', F|k)$ we have $\Phi(\sigma)(U_K) = U_{K'}$ if and only if $\sigma(K')^i = K^i$ (Galois correspondence), or equivalently if $\tilde{\sigma}(K') = K$ for some Frobenius twist $\tilde{\sigma} \sim \sigma$. 
\end{proof}

\begin{prop}
	\label{prop-injectivity}
	The map $\Phi$ of Theorem~\ref{thm-main-theorem} is injective.
\end{prop}

\begin{proof}
	If $\Phi(\sigma_1) = \Phi(\sigma_2)$, choose an arbitrary function field $K|k$ in $F$ and find $K'$ such that $\Phi(\sigma_i)(U_K) = U_{K'}$ (Galois correspondence). Then in the diagram above, $\sigma_1$ and $\sigma_2$ come from upstairs where $\Phi$ is injective by Theorem~\ref{thm-injectivity}. 
\end{proof}

\section{Detecting the Rationality of Function Fields}
\label{sec-detect-rationality}

The aim of this section is to prove the following Proposition~\ref{prop-rational-group-autom-invariant} and to show that the characteristic $\Char(k)$ is encoded in $G_{F|k}$ (Proposition~\ref{prop-char}).

\begin{prop}
	\label{prop-rational-group-autom-invariant}
	Given $\lambda: G_{F|k} \isomto G_{F'|k'}$, the bijection of Proposition~\ref{prop-function-field-bijection} maps rational function fields to rational function fields. 
\end{prop}

Note that the rationality of a function field depends only on its perfect equivalence class, for the function fields perfectly equivalent to $k(x)$ are given by $k(x)^{p^n} = k(x^{p^n})$ for $n \in \bZ$, where $\Char(k) = p > 0$. Our proof of Proposition~\ref{prop-rational-group-autom-invariant} is an adaption of \cite{rovinsky}, Lemma~3.1\,(1) that takes into account the possibility of positive characteristic.

\begin{defi}
	Let $G$ be a group and $n \in \bN$. An element $x \in G$ is called \defstyle{$n$-divisible} if $y^n = x$ for some $y \in G$. It is called \defstyle{infinitely $n$-divisible} if there exists a sequence $(x_1,x_2,\ldots)$ with $x_1 = x$ and $x_i = x_{i+1}^n$ for all $i \in \bN$. We say the group $G$ is \defstyle{$n$-divisible} if every element is so.
\end{defi}

Recall that a group is \defstyle{virtually abelian} if it contains an abelian subgroup of finite index.

\begin{lem}
	\label{lem-pgl-divisible}
	Let $k$ be an algebraically closed field. Then $\Aut(\bP_k^1)$ is not virtually abelian, it is $\ell$-divisible for every prime number $\ell \neq \Char(k)$ but not $p$-divisible if $\Char(k) = p > 0$.
\end{lem}

\begin{proof}
	Suppose $H \leq \Aut(\bP_k^1)$ has finite index. Then $H$ contains a nontrivial translation $\phi(z) = z+b$, $b \neq 0$, and a nontrivial homothety $\psi(z) = az$, $a \neq 1$. They do not commute since
	\begin{align*}
		\psi \circ \phi(z) &= az + ab,\\
		\phi \circ \psi(z) &= az + b,
	\end{align*}
	but $ab \neq b$. It is enough to show the $\ell$-divisibility of $\GL(2,k)$ as this property descends to the quotient $\Aut(\bP_k^1) = \PGL(2,k)$. Given $A \in \GL(2,k)$, we may assume it is in Jordan normal form
	$$A = \left(\begin{matrix}\lambda & 0 \\ 0 & \mu \end{matrix}\right) \quad \text{or} \quad A = \left(\begin{matrix}\lambda & 1 \\ 0 & \lambda \end{matrix}\right).$$
	Then a matrix $B \in \GL(2,k)$ with $B^\ell = A$ is given by
	\[B = \left(\begin{matrix}\lambda^{1/\ell} & 0 \\ 0 & \mu^{1/\ell} \end{matrix}\right) \quad \text{or} \quad B = \lambda^{1/\ell}\left(\begin{matrix}1 & \lambda^{-1}/\ell \\ 0 & 1 \end{matrix}\right), \]
	respectively. Suppose $\Char(k) = p > 0$ and assume for contradiction that there exists $\psi \in \Aut(\bP_k^1)$ such that $\psi^p(z) = z + 1$. If $P \in \bP_k^1$ is a fixed point of $\psi$, it is also a fixed point of $\psi^p$, hence $P = \infty$. As a Möbius transformation with $\infty$ as its only fixed point, $\psi$ is a translation, $\psi(z) = z + a$ for some $a \in k$. But then $\psi^p(z) = z + p a = z \neq z+1$, contradiction!
\end{proof}

\begin{lem}
	\label{lem-detect-rational-group}
	Let $K|k$ be a 1-dimensional function field with perfect closure $K^i$ and let $\ell \neq \Char(k)$ be a prime number. Then $K$ is rational if and only if the subgroup of $\Aut(K^i|k)$ generated by the infinitely $\ell$-divisible elements is not virtually abelian.
\end{lem}

\begin{proof}
	By Lemma~\ref{lem-aut-Ki}, the infinitely $\ell$-divisible elements of $\Aut(K^i|k)$ are in fact infinitely $\ell$-divisible elements in the subgroup $\Aut(K|k)$. Moreover, if $C$ is a complete nonsingular model of $K|k$, we have isomorphisms 
	$$\Aut(K|k) \cong \Aut(C)^\op \cong \Aut(C).$$
	Thus, we define $H$ as the subgroup of $\Aut(C)$ generated by the infinitely $\ell$-divisible elements and show that $C$ is rational if and only if $H$ is not virtually abelian.
	
	If $C \cong \bP^1$ is rational, we have $H = \Aut(C) \cong \Aut(\bP_k^1)$ and this is not virtually abelian by Lemma~\ref{lem-pgl-divisible}. 	If $C$ has genus 1, it is a principal homogeneous space under the elliptic curve $E = \Jac C$ and we have $\Aut(C) \cong E(k) \rtimes \Aut(E)$. The abelian subgroup $E(k)$ is divisible, hence contained in $H$, so that $H$ is virtually abelian, $\Aut(E)$ being finite. For $C$ of higher genus, $\Aut(C)$ is finite and $H$ is virtually abelian via the trivial subgroup.
\end{proof}

\begin{proof}[Proof of Proposition~\ref{prop-rational-group-autom-invariant}]
	Let $K|k$ and $K'|k'$ be function fields with $\lambda^{-1}(U_{K'}) = U_K$. Then $\lambda$ induces an isomorphism 
	$$N_{G_{F|k}}(U_K)/U_K \isomto N_{G_{F'|k'}}(U_{K'})/U_{K'},$$
	hence $\Aut(K^i|k) \cong \Aut(K'^i|k')$. Now use Lemma~\ref{lem-detect-rational-group} with $\ell$ a prime number $\neq \Char(k), \Char(k')$ to test the rationality of $K$ (resp.~$K')$ by means of these automorphism groups.
\end{proof}

\begin{prop}
	\label{prop-char}
	For $F|k$ and $F'|k'$ as in Theorem~\ref{thm-main-theorem}, if $G_{F|k} \cong G_{F'|k'}$, then $\Char(k) = \Char(k')$.
\end{prop}

\begin{proof}
	Choose an arbitrary rational function field $K|k$ in $F$ and $K'$ with $\lambda^{-1}(U_{K'}) = U_K$. Then $\Aut(K^i|k) \cong \Aut(K'^i|k')$ as above and passing to the subgroups generated by the infinitely $\ell$-divisible elements, $\Aut(\bP_{k}^1) \cong \Aut(\bP_{k'}^1).$
	By Lemma~\ref{lem-pgl-divisible}, $\Char(k)$ is the unique prime $p$ for which $\Aut(\bP_k^1)$ is not $p$-divisible, or zero if no such prime exists, so it is the same for $k$ and $k'$.
\end{proof}

\section{Detecting Decomposition Groups}
\label{sec-dec-subgroups}
\label{sec-detect-dec-subgroups}

Our next aim is to give a group-theoretic characterisation in terms of $(G_{F|k}, U_K)$ of the decomposition groups in the pro-$\ell$ abelian Galois group $U_K^{\ab,\ell}$ (Proposition~\ref{prop-autom-preserve-dec-subgroups}). We start by recollecting some generalities. Let $k$ be an algebraically closed field, $\ell \neq \Char(k)$ a fixed prime number and $K|k$ the function field of a complete nonsingular curve $C$ over $k$. The normalised discrete valuations on $K|k$ correspond bijectively to closed points of $C$ and we write $\ord_P$ for the valuation associated with $P \in C(k)$. Denote by $K^{\ab,\ell}$ the maximal pro-$\ell$ abelian extension of $K$, obtained by adjoining the $\ell^n$th roots of all elements of $K$ for all $n \in \bN$. The Galois group $\Gal(K^{\ab,\ell}|K)$ acts transitively on the set of valuations of $K^{\ab,\ell}$ extending $\ord_P$ and their common stabiliser is the (pro-$\ell$) \defstyle{decomposition group} of $P$, which we denote by $Z_P$. Let $Z_{\tot}(C) \subseteq \Gal(K^{\ab,\ell}|K)$ be the closed subgroup topologically generated by all $Z_P$, called the (pro-$\ell$) \defstyle{total decomposition group} of $K|k$. Its fixed field is the maximal completely split pro-$\ell$ abelian extension of $K$, or equivalently the maximal unramified pro-$\ell$ abelian extension since all residue field extensions are trivial. Thus, we have an exact sequence
\begin{equation}
	\label{eq-total-dec-sequence}
	1 \lto Z_{\tot}(C) \lto \Gal(K^{\ab,\ell}|K) \lto \pi_1^{\ab,\ell}(C) \lto 1
\end{equation}
where $\pi_1^{\ab,\ell}(C)$ is the pro-$\ell$ abelianisation of the algebraic fundamental group of $C$. 

Denote by $\mu_{\ell^n}(k)$ the group of $\ell^n$th roots of unity in $k$ and by $\bZ_\ell(1) = \varprojlim_n \mu_{\ell^n}(k)$ the $\ell$-adic Tate module of $k^\times$. By Kummer theory, there is a natural isomorphism of topological groups
$$\Gal(K^{\ab,\ell}|K) \cong \Hom(K^\times, \bZ_\ell(1)).$$

\begin{defi}
	Write $\Map(C,\bZ_\ell(1))/\bZ_\ell(1)$ for the group of set-theoretic functions $C(k) \to \bZ_\ell(1)$ modulo the subgroup of constant functions. A function $f: C(k) \to \bZ_\ell(1)$ is a \defstyle{1-point function} at $P \in C(k)$ if it is constant on $C(k)\setminus\{P\}$. 
\end{defi}

Our 1-point functions are called "$\delta$-functions" in \cite{rovinsky}. Note that the 1-point functions at $P$ are closed under addition of constants, so the notion makes sense even for elements of $\Map(C, \bZ_{\ell}(1))/\bZ_{\ell}(1)$. A typical 1-point function at $P$ has the form $\delta_P\, \omega$ with $\omega \in \bZ_\ell(1)$, where $\delta_P: C(k) \to \{0,1\}$ is the Kronecker delta function. We recall the following well-known description of decomposition groups.

\begin{prop}
	\label{prop-dec-subgroups}
	There are canonical isomorphisms
	\begin{enumerate}[label=(\alph*)]
		\item{$Z_P \cong \bZ_\ell(1)$, for all closed points $P \in C$,}
		\item{$Z_{\tot}(C) \cong \Map(C,\bZ_\ell(1))/\bZ_\ell(1) \cong \Hom(\Div^0(C), \bZ_\ell(1))$,}
	\end{enumerate}
	under which the inclusion $Z_P \subseteq Z_{\tot}(C)$ is isomorphic to $\bZ_\ell(1) \overset{\delta_P \cdot{\;\;}}{\lto} \Map(C,\bZ_\ell(1))/\bZ_\ell(1)$ with image the 1-point functions at $P$.
\end{prop}

\begin{proof}
	Let $K_P$ be the henselisation of $K$ with respect to $\ord_P$, let $\OO_{C,P}^h$ be its valuation ring and $K_P^{\ab,\ell}$ its maximal pro-$\ell$ abelian extension, which is also the henselisation of $K^{\ab,\ell}$ with respect to a valuation extending $\ord_P$.
	The inclusion $Z_P \subseteq \Gal(K^{\ab,\ell}|K)$ is isomorphic to $\Hom(K_P^\times, \bZ_\ell(1))\hookrightarrow \Hom(K^\times, \bZ_\ell(1))$ by functoriality of Kummer theory. We have a split exact sequence
	\begin{diagram}
		1 \rar & {\OO_{C,P}^{h,\times}} \rar & K_P^\times \rar[swap]{\ord_P} & \bZ \rar \lar[bend right] & 1
	\end{diagram}
	and obtain another short exact sequence upon taking $\Hom(-,\bZ_\ell(1))$. The group $\OO_{C,P}^{h,\times}$ is $\ell$-divisible by Hensel's lemma, hence $\Hom(\OO_{C,P}^{h,\times}, \bZ_\ell(1)) = 0$ and we get the canonical isomorphism 
	$$\bZ_\ell(1) \cong \Hom(K_P^\times,\bZ_\ell(1)),$$
	given by $\omega \mapsto [x \mapsto \ord_P(x) \omega]$.
	Consider the exact sequence
	$$K^\times \overset{\divisor}{\lto} \Div^0(C) \lto \Pic^0(C) \lto 1.$$
	We have $\Hom(\Pic^0(C),\bZ_\ell(1)) = 0$ as $\Pic^0(C)$ is $\ell$-divisible, hence
	$$\divisor^*: \Hom(\Div^0(C), \bZ_\ell(1)) \longhookrightarrow \Hom(K^\times, \bZ_\ell(1)).$$
	is injective. The groups are endowed with the compact-open topology and $\divisor^*$ is a topological embedding as it is continuous and the groups are compact Hausdorff. We have an isomorphism 
	$$\Hom(\Div^0(C),\bZ_\ell(1)) \cong \Map(C,\bZ_\ell(1))/\bZ_\ell(1)$$
	induced by $\Hom(\Div(C),\bZ_\ell(1)) \cong \Map(C,\bZ_\ell(1))$. The maps 
	$$\bZ_\ell(1) \cong \Hom(K_P^\times, \bZ_\ell(1)) \hookrightarrow \Hom(K^\times, \bZ_\ell(1))$$
	all factor through $\bZ_\ell(1) \overset{\delta_P \cdot{\;\;}}{\lto} \Map(C,\bZ_\ell(1))/\bZ_\ell(1)$ and (b) follows since the images generate $\Map(C,\bZ_\ell(1))/\bZ_\ell(1)$ topologically.
\end{proof}

We are interested in the functorial behaviour of decomposition groups. Let $\phi: C_2 \to C_1$ be a dominant morphism between complete nonsingular curves over $k$ and $\phi^*: K_1 \hookrightarrow K_2$ the corresponding function field extension. Embed $K_1$ and $K_2$ in a common algebraic closure $F$ and let $U_{K_i} = \Aut(F|K_i)$. The inclusion $U_{K_2} \subseteq U_{K_1}$ induces two homomorphisms in opposite directions between the pro-$\ell$ abelianisations: 
\begin{itemize}
	\item{the \defstyle{corestriction} $\cores: U_{K_2}^{\ab,\ell} \to U_{K_1}^{\ab,\ell}$; and}
	\item{the \defstyle{restriction} $\res: U_{K_1}^{\ab,\ell} \to U_{K_2}^{\ab,\ell}$, also called the \defstyle{transfer} $\tr$.}
\end{itemize}
They come from profinite group homology via the isomorphisms $U_{K_i}^{\ab,\ell} \cong H_1(U_{K_i}, \bZ_\ell)$. The corestriction is also the natural map from the functoriality of $(-)^{\ab,\ell}$. Their effect on decomposition groups can be expressed via the pullback and pushward map of divisors.

\begin{prop}
	\label{prop-two-squares}
	The following squares commute:
	\begin{center}
		\begin{minipage}[t]{0.44\textwidth}
			\begin{tikzcd}
				\Hom(\Div^0(C_2),\bZ_\ell(1)) \rar \dar{- \circ \phi^*} & U_{K_2}^{\ab,\ell} \dar{\cores} \\
				\Hom(\Div^0(C_1),\bZ_\ell(1)) \rar & U_{K_1}^{\ab,\ell}
			\end{tikzcd}
		\end{minipage}
		\begin{minipage}[t]{0.44\textwidth}
			\begin{tikzcd}
				\Hom(\Div^0(C_2),\bZ_\ell(1)) \rar & U_{K_2}^{\ab,\ell} \\
				\Hom(\Div^0(C_1),\bZ_\ell(1)) \rar \uar{\deg_i(\phi)^{-1}\cdot(- \circ \phi_*)} \rar & U_{K_1}^{\ab,\ell} \uar{\tr}
			\end{tikzcd}
		\end{minipage}
	\end{center}
\end{prop}

Here $\deg_i(\phi) = [K_2 : K_1]_i$ denotes the inseparable degree, which is either $1$ or a power of the characteristic of $k$ so that multiplication by $\deg_i(\phi)^{-1}$ on $\bZ_\ell(1)$ is well-defined.

\begin{proof}
	From Kummer theory and the fact that corestriction is Pontryagin dual to restriction we have two commutative squares.
	\begin{center}
		\begin{minipage}[t]{0.40\textwidth}
			\begin{tikzcd}
				\Hom((K_2^i)^\times, \bZ_\ell(1)) \rar{\sim} \dar[swap]{- \circ \res} & U_{K_2}^{\ab,\ell} \dar[swap]{\cores} \\
				\Hom((K_1^i)^\times, \bZ_\ell(1)) \rar{\sim} & U_{K_1}^{\ab,\ell}
			\end{tikzcd}
		\end{minipage}
		\begin{minipage}[t]{0.48\textwidth}
			\begin{tikzcd}
				\Hom((K_2^i)^\times, \bZ_\ell(1)) \rar{\sim} & U_{K_2}^{\ab,\ell}  \\
				\Hom((K_1^i)^\times, \bZ_\ell(1)) \rar{\sim} \uar{- \circ \cores} & U_{K_1}^{\ab,\ell}. \uar{\res}
			\end{tikzcd}
		\end{minipage}
	\end{center}
	
	The restriction $\res: (K_1^i)^\times \to (K_2^i)^\times$ is the extension of the inclusion $\phi^*: K_1^\times \hookrightarrow K_2^\times$ and the corestriction $\cores: (K_2^i)^{\times} \to (K_1^i)^{\times}$ is the group-theoretic norm
	$$\cores(x_2) = \prod_{\sigma \in U_{K_1}/U_{K_2}} \sigma x_2,$$
	which on $K_2^\times$ is related to the field-theoretic norm $\Nm: K_2^\times \to K_1^\times$ by $\cores(x_2) = \Nm(x_2)^{\deg_i(\phi)}$. The claim follows now from the formulae for the divisor map
	\begin{align*}
		\divisor(\phi^*x_1) &= \phi^*(\divisor(x_1)),\\
		\divisor(\Nm \,x_2) &= \phi_*(\divisor(x_2)). \qedhere
	\end{align*}
\end{proof}

Given $\phi: C_2 \to C_1$ as above, we use the terms corestriction and transfer not just for the homomorphisms between the pro-$\ell$ abelianised absolute Galois groups but also for the corresponding maps between the groups $\Hom(\Div^0(C_i), \bZ_\ell(1))$. On 1-point functions and hence on decomposition groups they act as follows. Given $P \in C_2(k)$, $Q \in C_1(k)$ and $\omega \in \bZ_\ell(1)$, we have
\begin{align}
	\cores(\delta_P \,\omega) &= e_\phi(P)\, \delta_{\phi(P)} \, \omega, \label{eq-cor-delta}\\
	\tr(\delta_Q\, \omega) &= \deg_i(\phi)^{-1}\, \mathbbm{1}_{\phi^{-1}(Q)} \, \omega, \label{eq-tr-delta}
\end{align}
where $e_\phi(P)$ denotes the ramification index of $\phi$ at $P$ and $\mathbbm{1}_{\phi^{-1}(Q)}$ denotes the characteristic function of $\phi^{-1}(Q)$ on $C_2(k)$. 

With these generalities at hand, we turn to the proof of the following.

\begin{prop}
	\label{prop-autom-preserve-dec-subgroups}
	In the situation of Theorem~\ref{thm-main-theorem}, given $\lambda: G_{F|k} \cong G_{F'|k'}$, let $K|k$ and $K'|k'$ be corresponding function fields with complete nonsingular models $C,C'$. Then the isomorphism $\lambda^{\ab,\ell}: U_K^{\ab,\ell} \isomto U_{K'}^{\ab,\ell}$ maps decomposition groups to decomposition groups, inducing a bijection
	$$\lambda^*: C'(k') \isomto C(k).$$
\end{prop}

We first treat the case of rational function fields and then extend to the general case. Let $K|k$ be a rational function field in $F$ with complete nonsingular model $C$. Let $H \subseteq N_{G_{F|k}}(U_K)/U_K$ be the subgroup generated by the infinitely $\ell$-divisible elements. Recall from the proof of Lemma~\ref{lem-detect-rational-group} the isomorphism $H \cong \Aut(K|k) \cong \Aut(C)^{\op}$. For $\sigma \in H$, denote by $\sigma^*$ the corresponding automorphism of $C$ and by $\ad(\sigma)$ the automorphism of $U_K^{\ab,\ell}$ induced by conjugation. For $P \in C(k)$, denote by $\Stab_H(P)$ the stabiliser
$$\Stab_H(P) = \{ \sigma \in H : \sigma^*(P) = P \}.$$

\begin{lem}
	\label{lem-dec-group-fixed-group}
	An element of $U_K^{\ab,\ell}$ belongs to the decomposition subgroup $Z_P$ if and only if it is fixed by $\ad(\sigma)$ for all $\sigma \in \Stab_H(P)$. 
\end{lem}

\begin{proof}
	Since the homomorphism $\Map(C,\bZ_\ell(1))/\bZ_\ell(1) \to U_K^{\ab,\ell}$ from Proposition~\ref{prop-dec-subgroups} depends functorially on $F|K|k$ with respect to isomorphisms, any $\sigma \in H$ induces a commutative square
	\begin{center}
		\begin{tikzcd}
			\Map(C,\bZ_\ell(1))/\bZ_\ell(1) \dar{- \circ \sigma^*} \rar & U_K^{\ab,\ell} \dar{\ad(\sigma)} \\
			\Map(C,\bZ_\ell(1))/\bZ_\ell(1) \rar & U_K^{\ab,\ell}.
		\end{tikzcd}
	\end{center}
	By the exact sequence~\ref{eq-total-dec-sequence}, the horizontal maps are injective with cokernel $\pi_1^{\ab,\ell}(C) = 1$, so they are in fact isomorphisms. We have to show that $f: C(k) \to \bZ_\ell(1)$ is a 1-point function at $P$ if and only if $f \circ \sigma^* \equiv f$ mod constants for all $\sigma \in \Stab_H(P)$. As these $\sigma^*$ fix $P$, the congruence mod constants amounts to equality. Now $\Stab_H(P)$ acts transitively on $C(k)\setminus\{P\}$, therefore if $f \circ \sigma^* = f$ for all $\sigma \in \Stab_H(P)$, then $f$ is constant on $C(k) \setminus \{P\}$. 
\end{proof}

We are led to the question of a group-theoretic characterisation of the stabiliser subgroups in $\Aut(\bP^1)$. Recall that an automorphism of $\bP^1$ is \defstyle{parabolic} if it has a single fixed point. E.\,g., the parabolic elements with fixed point $\infty$ are the non-trivial translations $z \mapsto z+b$ with $b \in k$. 

\begin{lem}
	\label{lem-stab-subgroups-in-pgl}
	\begin{enumerate}[leftmargin=1.5cm,label=(\alph*),itemsep=0ex]
		\item{An element of $\Aut(\bP^1)$ is parabolic if and only if it is uniquely $\ell$-divisible.}
		\item{A subgroup of $\Aut(\bP^1)$ is the stabiliser subgroup of a point $P \in \bP^1(k)$ if and only if it is the normaliser of the centraliser of a parabolic element.}
	\end{enumerate}
\end{lem}

\begin{proof}
	The parabolic elements constitute the conjugacy class of the translation $z+1$. This is uniquely $\ell$-divisible with $z+1/\ell$ as the unique dividing element. Indeed, if $\varphi^\ell = z+1$, then $\varphi$ commutes with $z+1$, which implies that $\varphi$ is also parabolic with fixed point $\infty$, i.e.\ a translation, necessarily equal to $\varphi(z) = z +1/\ell$. The non-parabolic elements are conjugate to a homothety $az$ and are non-uniquely $\ell$-divisible as there are $\ell$ distinct roots $\sqrt[\ell]{a}$.
	
	For (b) it suffices to show that the stabiliser of $\infty$, which consists of the affine transformations $az+b$ with $a \in k^\times$ and $b \in k$, is the normaliser of the centraliser of $z+1$. The centraliser of $z+1$ is the group of translations $z+b$. These are indeed normalised by the affine transformations and conversely, if $\psi \circ (z+1) \circ \psi^{-1}$ is a translation, we have $\psi(\infty) = \infty$ since $\psi$ maps the fixed point of $z+1$ to that of $\psi \circ (z+1) \circ \psi^{-1}$.
\end{proof}

The two preceding lemmas prove Proposition~\ref{prop-autom-preserve-dec-subgroups} in the rational case. For the general case, call a compact open subgroup $V \subseteq G_{F|k}$ \defstyle{rational} if $V = U_{k(x)}$ for some rational function field $k(x)$ in $F$. Recall that the rational compact open subgroups are group-theoretically distinguished by Proposition~\ref{prop-rational-group-autom-invariant}.

\begin{lem}
	\label{lem-detect-dec-groups}
	Let $K|k$ be a function field in $F$ with complete nonsingular model $C$. A subgroup $Z \subseteq U_K^{\ab,\ell}$ is the decomposition group of a point $P \in C(k)$ if and only if the following hold:
	\begin{enumerate}[leftmargin=1.5cm,label=(\alph*),itemsep=0ex]
		\item{There exists a rational $V \supseteq U_K$ such that $Z$ is the image of a decomposition group under the transfer map $\tr: V^{\ab,\ell} \to U_K^{\ab,\ell}$.}
		\item{For every rational $V \supseteq U_K$, the image of $Z$ under the corestriction map $\cores: U_K^{\ab,\ell} \to V^{\ab,\ell}$ is contained in a decomposition group.}
	\end{enumerate}
\end{lem}

\begin{proof}
	Assume that $Z = Z_P$ is a decomposition group. By Riemann-Roch, there exists a morphism $\phi: C \to \bP^1$ having a pole at $P$ and no other poles. Let $k(x) \subseteq K$ be the corresponding inclusion of function fields in $F$ and $V = U_{k(x)} \supseteq U_K$. For a 1-point function $\delta_{\infty} \,\omega$ at $\infty \in \bP^1$, we have
	$$\tr(\delta_\infty\,\omega) = \deg_i(\phi)^{-1} \mathbbm{1}_{\phi^{-1}(\infty)} \, \omega = \deg_i(\phi)^{-1} \delta_P \, \omega,$$
	which shows $\tr(Z_\infty) = Z_P$. For (b), let $V \subseteq U_K$ be any rational group and $\phi: C \to \bP^1$ a corresponding morphism. Then for $P \in C(k)$, we have
	$$\cores(\delta_P\,\omega) = e_\phi(P) \delta_{\phi(P)} \,\omega,$$
	hence $\cores(Z_P) = e_\phi(P) Z_{\phi(P)}$ is contained in the decomposition group $Z_{\phi(P)}$.
	
	Suppose conversely that $Z \subseteq U_K^{\ab,\ell}$ satisfies the two conditions. The transfer calculation above shows that (a) is equivalent to the existence of a non-empty finite subset $S \subseteq C(k)$ (namely a fibre of a morphism $C \to \bP^1$) such that the elements of $Z$ correspond to the functions $f: C(k) \to \bZ_\ell(1)$ that are constant on $S$ and vanish on $C(k) \setminus S$. Let $\phi: C \to \bP^1$ be a morphism separating the points in $S$ and let $V \supseteq U_K$ be the corresponding rational group. For $f = \mathbbm{1}_S\, \omega$ we have
	$$
	\cores(f)(Q) = (f \circ \phi^*)(Q) = 
	\begin{cases}
	e_\phi(P)\,\omega & \text{if } \phi^{-1}(Q) \cap S = \{P\},\\
	0 & \text{if } \phi^{-1}(Q) \cap S = \emptyset.
	\end{cases}
	$$
	By (b), this should be a 1-point function at a point, hence $S$ consists just of a single point, $S = \{P\}$, and $Z = Z_P$ is the decomposition group of $P$.
\end{proof}

\section{Proof of Main Theorem}
\label{sec-main-result}

\begin{lem}
	\label{lem-pullback}
	In the situation of Theorem~\ref{thm-main-theorem}, given $\lambda: G_{F|k} \isomto G_{F'|k'}$, let $K_1 \subseteq K_2$ and $K_1' \subseteq K_2'$ be extensions of function fields with $\lambda^{-1}(U_{K_i'}) = U_{K_i}$, corresponding to morphisms $\phi: C_2 \to C_1$ and $\phi': C_2' \to C_1'$. Let $\lambda^*: \Div(C_i') \to \Div(C_i)$ be the linear extension of the bijection from Proposition~\ref{prop-autom-preserve-dec-subgroups}. Of the two squares
	\begin{center}
		\begin{minipage}[t]{0.40\textwidth}
			\begin{diagram}
				\Div(C_2') \rar{\lambda^*}[below]{\sim} \dar{\phi'_*} & \Div(C_2) \dar{\phi_*} \\
				\Div(C_1') \rar{\lambda^*}[below]{\sim} & \Div(C_1),
			\end{diagram}
		\end{minipage}
		\begin{minipage}[t]{0.40\textwidth}
			\begin{diagram}
				\Div(C_2') \rar{\lambda^*}[below]{\sim} & \Div(C_2) \\
				\Div(C_1') \rar{\lambda^*}[below]{\sim} \uar{\phi'^*} & \Div(C_1) \uar{\phi^*},
			\end{diagram}
		\end{minipage}
	\end{center}
	the first always commutes and the second commutes if $\deg_i(\phi) = \deg_i(\phi')$.
\end{lem}

\begin{proof}
	By equations~\eqref{eq-cor-delta} and \eqref{eq-tr-delta}, the corestriction and transfer maps restrict for $Q \in C_1$ as follows
	\begin{diagram}
		\displaystyle\bigoplus_{P \in \phi^{-1}(Q)} Z_P \rar[hook] \dar[xshift=-.3em,swap]{\cores} & U_{K_2}^{\ab,\ell} \dar[xshift=-.3em, swap]{\cores} \\
		Z_Q \rar[hook] \uar[xshift=.3em,swap]{\tr} & U_{K_1}^{\ab,\ell}. \uar[xshift=.3em,swap]{\tr} 
	\end{diagram}
	For $P \in \phi^{-1}(Q)$, we have $\cores(Z_P) = e_\phi(P)Z_Q \subseteq Z_Q$ of finite index. This determines $Q = \phi(P)$ uniquely since decomposition groups of different points have trivial intersection, hence the commutativity of the first square. We have
	$$(\tr\circ\cores)(\delta_P\,\omega) = \frac{e_\phi(P)}{\deg_i(\phi)} \mathbbm{1}_{\phi^{-1}(Q)}\,\omega = \frac{e_\phi(P)}{\deg_i(\phi)} \sum_{P' \in \phi^{-1}(Q)}\delta_{P'}\,\omega,$$
	thus the endomorphism of $Z_P$ induced by $\tr \circ \cores$ is multiplication by $e_\phi(P)/\deg_i(\phi)$. Therefore, assuming equal inseparable degrees of $\phi$ and $\phi'$, the ramification indices match and the second square commutes.
\end{proof}

\begin{rmk}
	The $\ell$-part of the ramification index $e_\phi(P)$ can also be reconstructed as the index $(Z_Q : \cores(Z_P))$. If one restricts oneself to fields of characteristic zero, $\ell$ can be an arbitrary prime and this would give an alternative reconstruction of ramification indices.
\end{rmk}

The \defstyle{projectivisation} of a $k$-vector space $V$ is the set $\bP_k V = (V \setminus \{0\})/k^\times,$
together with the projective lines as distinguished subsets. A map $\bP_{k'} V' \hookrightarrow \bP_k V$ is a \defstyle{projective embedding} if it is injective and maps lines onto lines. It is a \defstyle{collineation} if it admits an inverse projective embedding. Given a function field $K|k$ with complete nonsingular model $C$, we view $K^\times/k^\times = \bP_k K$ as a projective space over $k$, and with it the group $\PDiv(C) \cong K^\times/k^\times$ of principal divisors. The strategy is then to recover the projective structure on $\PDiv(C)$ and reconstruct the function field $K|k$ by an application of the fundamental theorem of projective geometry. This idea appears already in~\cite{bogomolov} and later in~\cite{bogomolov-tschinkel-reconstruction} (Theorem~3.6) and~\cite{pop-recovering-function-fields}. We shall use the following slightly more general form for projective embeddings rather than collineations.

\begin{thm}[Fundamental Theorem of Projective Geometry, \cite{artin-ga}, Thm~II.2.26]
	Let $k$ and $k'$ be arbitrary fields, let $V$ and $V'$ be vector spaces of dimension $\geq 3$ over $k$ and $k'$, respectively, and let $\varphi: \bP_{k'}V' \longhookrightarrow \bP_k V$ be a projective embedding.
	Then there exist a field isomorphism $\tau: k' \isomto k$ and a $\tau$-semilinear injection $\Phi: V' \hookrightarrow V$ lifting $\varphi$, i.e.\ 
	$$\varphi(v' \bmod k'^\times) = \Phi(v') \bmod k^\times \quad \text{for } v' \in V'\setminus\{0\}.$$
	If $(\tilde \tau, \tilde \Phi)$ is another such pair, then $\tilde \tau = \tau$ and $\tilde \Phi = m_\alpha \circ \Phi$ for a unique $\alpha \in k^\times$, where $m_\alpha \in \Aut_{k}(V)$ is multiplication by $\alpha$. \qed
\end{thm}

\begin{lem}
	\label{lem-collineation-from-field-isom}
	Let $k$ and $k'$ be arbitrary fields, $K|k$ and $K'|k'$ two field extensions of degree $\geq 3$ and suppose that
	$$\varphi: K'^\times/k'^\times \longhookrightarrow K^\times/k^\times$$
	is simultaneously a projective embedding and a homomorphism of abelian groups. Then there exists a unique homomorphism of field extensions 
	$$\Phi: K'|k' \longhookrightarrow K|k$$
	lifting $\varphi$ and restricting to an isomorphism $k' \cong k$.
\end{lem}

\begin{proof}
	Let $(\tau,\Phi)$ be the pair from the fundamental theorem of projective geometry, which is uniquely determined by requiring $\Phi(1) = 1$. We have to show that $\Phi$ respects multiplication. Fix $x' \in K'^\times$ and let $m_{x'}: K' \to K'$ and $m_{\Phi(x')}: K \to K$ be multiplication by $x'$ and $\Phi(x')$, respectively. We need to show that the two maps 
	$$\Phi \circ m_{x'} \;\text{ and }\: m_{\Phi(x')} \circ \Phi: K' \to K$$
	are equal. They are both $\tau$-semilinear and by multiplicativity of $\varphi$ induce the same projective embedding $K'^\times/k'^\times \to K^\times/k^\times$. By the uniqueness statement in the fundamental theorem of projective geometry, there exists a unique $\alpha \in k^\times$ such that $\Phi \circ m_{x'} = m_\alpha \circ m_{\Phi(x')} \circ \Phi$. The normalisation $\Phi(1) = 1$ forces $\alpha = 1$, so the two maps are equal.
\end{proof}

\begin{prop}
	\label{prop-isom-and-collineation}
	In the situation of Theorem~\ref{thm-main-theorem}, given $\lambda: G_{F|k} \isomto G_{F'|k'}$, let $K|k$ and $K'|k'$ be corresponding function fields. Then $\lambda$ induces an isomorphism
	$$\lambda^*: K'^\times/k'^\times \isomto K^\times/k^\times$$
	of abelian groups which is simultaneously a collineation of projective spaces. If $K_1 \subseteq K_2$ and $K_1' \subseteq K_2'$ have the same inseparable degree, the following square commutes:
	\begin{diagram}
		K_1'^\times/k'^\times \dar[hook] \rar{\lambda^*}[swap]{\sim} & K_1^\times/k^\times \dar[hook] \\
		K_2'^\times/k'^\times \rar{\lambda^*}[swap]{\sim} & K_2^\times/k^\times. 
	\end{diagram}
\end{prop}

\begin{proof}
	A divisor $D \in \Div(C)$ is principal if and only if there exist a morphism $\phi: C \to \bP^1$ and two points $Q_0,Q_1 \in \bP^1(k)$ such that $D = \phi^*(Q_0) - \phi^*(Q_1)$, thus by Lemma~\ref{lem-pullback} the isomorphism $\lambda^*: \Div(C') \cong \Div(C)$ restricts to the subgroups of principal divisors and induces $\lambda^*$ as claimed. The lines in $\PDiv(C)$ are given by
	$$D + \left\{ \phi^*(Q) - \phi^*(Q_0) \suchthat Q \in \bP^1(k) \right\}$$
	for a principal divisor $D \in \PDiv(C)$ , a morphism $\phi: C \to \bP^1$ and a point $Q_0 \in \bP^1(k)$. Indeed, a line in $\PDiv(C)$ corresponds to a 2-dimensional $k$-subspace $\langle f, g\rangle \subseteq K$, and if $\phi: C \to \bP^1$ is the morphism given by $f/g$, the line is
	\begin{align*}
	\left\{ \divisor(af+bg) \suchthat (a:b) \in \bP^1 \right\} &= \divisor(g) +  \left\{ \divisor(af/g+b) \suchthat (a:b) \in \bP^1 \right\} \\
	&= \divisor(g) + \left\{ \phi^*(-b:a) - \phi^*(\infty) \suchthat (a : b) \in \bP^1 \right\}.
	\end{align*}
	Thus, again by Lemma~\ref{lem-pullback}, $\lambda^*$ is a collineation. The commutativity of the square follows from the same lemma.
\end{proof}

Hence in the situation of Theorem~\ref{thm-main-theorem}, given $\lambda: G_{F|k} \isomto G_{F'|k'}$, we have for every pair of corresponding function fields $K|k$ and $K'|k'$ an isomorphism 
\begin{align}
	\sigma_{K,K'}: K'|k' \isomto K|k. \label{eq-sigma-isom}
\end{align}

\begin{lem}
	\label{lem-choose-K'}
	Given $\lambda: G_{F|k} \isomto G_{F'|k'}$, there exists a map $K \mapsto K'$ from the set of function fields in $F$ to the set of function fields in $F'$ such that $\lambda^{-1}(U_{K'}) = U_K$ for all $K$, and whenever $K_1 \subseteq K_2$, then $K_1' \subseteq K_2'$ of the same inseparable degree. 
\end{lem}

\begin{proof}
	Assume $\Char(k) = p > 0$. Fix a function field $K_0$ in $F$ and choose any corresponding $K_0'$. For every $K  \supseteq K_0$, choose $K'$ in its perfect equivalence class such that $[K' : K_0']_i = [K : K_0]_i$. Then for an arbitrary function field $K$ in $F$, choose $K'$ in its perfect equivalence class such that $[(K K_0)' : K']_i = [K K_0 : K]_i$. One verifies the assertions by looking at the field diagram
	
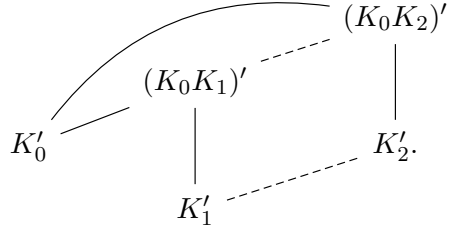
\begin{center}
		\begin{tikzcd}[row sep = tiny]
			& & (K_0 K_2)' \arrow[dash]{dd} \\
			& (K_0 K_1)' \urar[dashed,-] \arrow[dash]{dd} &  \\
			K_0' \urar[dash] \arrow[uurr, bend left, -] &  & K_2'. \\
			& K_1' \urar[dashed,-] & 
		\end{tikzcd}
	\end{center}
	The fields are chosen for the solid lines to have inseparable degrees matching those of the corresponding extensions in $F$. It follows that the same holds for the dashed lines.
\end{proof}

Choosing a map $K \mapsto K'$ according to the lemma, the isomorphisms $\sigma_{K,K'}$ from~(\ref{eq-sigma-isom}) are compatible with each other, hence define $\sigma: F'|k' \isomto F|k$ that satisfies $\sigma(K') = K$ for all function fields $K$ in $F$. It remains to show that the induced isomorphism $\Phi(\sigma): G_{F|k} \isomto G_{F'|k'}$ coincides with the given $\lambda$.

\begin{lem}
	\label{lem-autom-preserv-function-fields}
	Let $F|k$ be an extension of algebraically closed fields with $\trdeg(F|k) = 1$.
	\begin{enumerate}[label=(\alph*)]
		\item{Suppose $\sigma \in \Aut(F)$ satisfies $\sigma k = k$ and $\sigma K = K$ for all function fields $K$ in $F$. Then $\sigma = \id$.}
		\item{Assume $\Char(k) = p > 0$ and suppose $\sigma \in \Aut(F)$ satisfies $\sigma k = k$ and $\sigma K^i = K^i$ for all function fields $K|k$ in $F$. Then $\sigma$ is an integral power of the Frobenius automorphism.}
		\item{Let $\lambda \in \Aut(G_{F|k})$ be a topological automorphism such that $\lambda(U) = U$ for all compact open subgroups $U$ in $G_{F|k}$. Then $\lambda = \id$.}
	\end{enumerate}
\end{lem}

\begin{proof}
	\begin{enumerate}[label=(\alph*)]
		\item{
			Let $x \in F\setminus k$, use it as a coordinate on $\bP^1$. For every function field extension $k(x) \subseteq K$, corresponding to a morphism $\phi: C \to \bP^1$, the automorphism $\sigma$ permutes the normalised discrete valuations of $k(x)|k$ that ramify in $K$, i.e.\ the branch points of $\phi$. But every two-element subset of $\bP^1(k)$ is the branch locus of some $\phi$, so $\sigma$ must act trivially on the set of normalised discrete valuations of $k(x)|k$. Therefore we have $\sigma y / y \in k^\times$ for all $y \in k(x)^\times$ and we conclude $\sigma = \id$ on $k(x)$ by Lemma~\ref{lem-autom-id-citerion}. Since $x$ was arbitrary, $\sigma = \id$ on $F$.
		}
		\item{
			For each function field $K$, there exists a unique $n \in \bZ$ such that $\sigma K = K^{p^{n}}$. We claim that $n$ is independent of $K$. Indeed, if one function field is contained in another, $K_1 \subseteq K_2$, we find $n_1 = n_2$ by looking at generalised inseparable degrees of $K_1 \subseteq K_2$ and $\sigma K_1 \subseteq \sigma K_2$. The general case follows since any two function fields are contained in a common finite extension. Now $(\sigma \circ \Frob^{-n})(K) = K$ for all function fields $K$, thus $\sigma = \Frob^{n}$ by (a).
		}
		\item{
			Let $\sigma \in G_{F|k}$. For all compact open subgroups $U$ of $G_{F|k}$ we have
			$$\sigma U \sigma^{-1} = \lambda(\sigma U \sigma^{-1}) = \lambda(\sigma) U \lambda(\sigma)^{-1},$$
			thus $\sigma^{-1}\lambda(\sigma)$ is contained in the normaliser of $U$. This implies $\sigma^{-1}\lambda(\sigma)(K^i) = K^i$ for all function fields $K|k$ in $F$, so $\sigma^{-1}\lambda(\sigma)$ is the identity in characteristic $0$ and a power of the Frobenius in positive characteristic by (a) and (b), respectively. But it fixes $k$ elementwise, hence $\sigma^{-1}\lambda(\sigma) = 1$.
			\qedhere
		}
	\end{enumerate}
\end{proof}

For the isomorphism $\sigma: F'|k' \isomto F|k$ constructed from $\lambda: G_{F|k} \isomto G_{F'|k'}$, we have 
$$\Phi(\sigma)(U_K) = \sigma^{-1} U_K \sigma = U_{\sigma^{-1}K} = U_{K'} = \lambda(U_K)$$
for all function fields $K$ in $F$. Thus $\Phi(\sigma)^{-1} \circ \lambda$ satisfies the hypotheses of Lemma~\ref{lem-autom-preserv-function-fields}~(c) and we conclude $\lambda = \Phi(\sigma)$, finishing the proof of Theorem~\ref{thm-main-theorem}.

\bibliographystyle{amsalpha}
\bibliography{bibliography}

\end{document}